\documentclass[reqno,11pt]{amsart}

\usepackage[a4paper,left=25mm,right=25mm,top=30mm,bottom=30mm]{geometry}
\usepackage{amsmath}
\usepackage{amssymb}
\usepackage{amsthm}
\usepackage{amscd}
\usepackage{mathtools}
\usepackage{url}
\usepackage[ansinew]{inputenc}
\usepackage{color}
\usepackage[final]{graphicx}
\usepackage{esint} 
\usepackage{bbm}
\usepackage{subfigure}
\usepackage{tikz}
\usepackage{cite}
\usepackage{enumitem}
\usepackage[hidelinks]{hyperref}

\renewcommand{\epsilon}{\varepsilon}

\numberwithin{equation}{section}

\newtheoremstyle{thmlemcorr}{10pt}{10pt}{\itshape}{}{\bfseries}{.}{10pt}{{\thmname{#1}\thmnumber{ #2}\thmnote{ (#3)}}}
\newtheoremstyle{thmlemcorr*}{10pt}{10pt}{\itshape}{}{\bfseries}{.}\newline{{\thmname{#1}\thmnumber{ #2}\thmnote{ (#3)}}}
\newtheoremstyle{defi}{10pt}{10pt}{\itshape}{}{\bfseries}{.}{10pt}{{\thmname{#1}\thmnumber{ #2}\thmnote{ (#3)}}}
\newtheoremstyle{remexample}{10pt}{10pt}{}{}{\bfseries}{.}{10pt}{{\thmname{#1}\thmnumber{ #2}\thmnote{ (#3)}}}
\newtheoremstyle{ass}{10pt}{10pt}{}{}{\bfseries}{.}{10pt}{{\thmname{#1}\thmnumber{ A#2}\thmnote{ (#3)}}}

\theoremstyle{thmlemcorr}
\newtheorem{theorem}{Theorem}
\numberwithin{theorem}{section}
\newtheorem{lemma}[theorem]{Lemma}
\newtheorem{corollary}[theorem]{Corollary}
\newtheorem{proposition}[theorem]{Proposition}

\theoremstyle{thmlemcorr*}
\newtheorem{theorem*}{Theorem}
\newtheorem{lemma*}[theorem]{Lemma}
\newtheorem{corollary*}[theorem]{Corollary}
\newtheorem{proposition*}[theorem]{Proposition}
\newtheorem{problem*}[theorem]{Problem}
\newtheorem{conjecture*}[theorem]{Conjecture}

\theoremstyle{defi}

\theoremstyle{remexample}

\theoremstyle{ass}

\newcommand{\Bcal}{\mathcal{B}}

\newcommand{\Ecal}{\mathcal{E}}
\newcommand{\Fcal}{\mathcal{F}}

\newcommand{\Lcal}{\mathcal{L}}

\newcommand{\Qcal}{\mathcal{Q}}

\newcommand{\Ibb}{\mathbb{I}}

\DeclareMathOperator{\id}{id}

\DeclareMathOperator{\diverg}{div}

\DeclareMathOperator{\dist}{dist}

\DeclareMathOperator{\supp}{supp}

\newcommand{\absB}[1]{\Bigl|#1\Bigr|}

\newcommand{\N}{\mathbb{N}}
\newcommand{\R}{\mathbb{R}}

\newcommand{\ONE}{\mathbbm{1}}

\newcommand{\weakly}{\rightharpoonup}

\newcommand{\eps}{\epsilon}

    \newcommand{\vertiii}[1]{{\left\vert\kern-0.25ex\left\vert\kern-0.25ex\left\vert #1 
    \right\vert\kern-0.25ex\right\vert\kern-0.25ex\right\vert}}

\def\XXint#1#2#3{{\setbox0=\hbox{$#1{#2#3}{\int}$}
\vcenter{\hbox{$#2#3$}}\kern-.5\wd0}}

\DeclareMathOperator{\SO}{SO}

\DeclarePairedDelimiter\abs{\lvert}{\rvert}
\DeclarePairedDelimiter{\norm}{\lVert}{\rVert}

\renewcommand{\phi}{\varphi}

\newcommand{\dx}{\, dx}
\newcommand{\dy}{\, dy}

\newcommand{\dt}{\, dt}

\def\XXint#1#2#3{{\setbox0=\hbox{$#1{#2#3}{\int}$}
     \vcenter{\hbox{$#2#3$}}\kern-.5\wd0}}

\makeatletter
\g@addto@macro\bfseries{\boldmath}
\makeatother

\usepackage{esint}

\usepackage{appendix}

\def\XXint#1#2#3{{\setbox0=\hbox{$#1{#2#3}{\int}$}
		\vcenter{\hbox{$#2#3$}}\kern-.5\wd0}}

\definecolor{colora}{RGB}{21,128,66}
\definecolor{colorb}{RGB}{21,66,128}
\definecolor{colorc}{RGB}{128,21,66}

\usepackage{lmodern}
\usepackage{graphicx}
\usepackage{wrapfig}

\title[Nonlocal linearization]{Linearization in nonlocal elasticity via Gamma-convergence}

\author{Carolin Kreisbeck}
\address{Mathematisch-Geographische Fakult\"at, Katholische Universit\"at Eichst\"att-Ingolstadt, Osten\-stra{\ss}e 28, 85072 Eichst\"att, Germany}
\email{carolin.kreisbeck@ku.de}

\author{Carlos Mora-Corral}
\address{Departamento de Matem\'aticas, Universidad Aut\'onoma de Madrid, 28049 Madrid, Spain and Instituto de Ciencias Matem\'aticas, CSIC-UAM-UC3M-UCM, 28049 Madrid, Spain}
\email{carlos.mora@uam.es}

\author{Felix Seifert}
\address{Mathematisch-Geographische Fakult\"at, Katholische Universit\"at Eichst\"att-Ingolstadt, Osten\-stra{\ss}e 28, 85072 Eichst\"att, Germany}
\email{felix.seifert@ku.de}

\begin{document}
\maketitle

\thispagestyle{empty}
\begin{abstract}  
Linearization is a classical tool in continuum mechanics that connects nonlinear models of solids to their linear counterparts through the small-strain limit, as in the passage from finite to linear elasticity.  
Here, we study a nonlocal analogue of this limit process in the context of state-based peridynamics, starting from a class of models in nonlocal hyperelasticity subject to  nonlocal Dirichlet boundary conditions. Recent formal calculations in the literature indicate that the linearized versions of these models are closely related to the widely used Eringen models. 
 The main purpose of this note is to make this connection rigorous by establishing \(\Gamma\)-convergence of the associated energy functionals and a corresponding compactness result that together imply strong convergence of minimizers. 
 Our proof relies on a local-to-nonlocal transfer argument that reduces the problem to a well-known local linearization result. 

\vspace{8pt}

\vspace{8pt}

\vspace{8pt}

\noindent\textsc{Date:} \today.
\end{abstract}

\section{Introduction} 
Linearization is an integral concept in continuum mechanics. 
While nonlinear models are needed to capture the response of a body under large deformations, their simpler linearized versions can be efficiently applied to describe the regime of small strains. In the mechanics literature, linearized models are commonly deduced by formal calculations based on Taylor expansion. The first mathematical derivation via $\Gamma$-convergence was achieved by Dal Maso, Negri \& Percivale~\cite{DNP02} in the context of elasticity. 
These results have since been extended in several directions, for instance, by weakening the coercivity conditions \cite{ADD12}, by allowing for energy densities with multiple wells \cite{ADLP18} and, more recently, with infinitely many wells \cite{AlLaPaWo26}, by removing restrictions on the traction forces \cite{MaorMora21}, or by accounting for surface energy terms \cite{KM25}. 
Similar results are also available for other material models,  e.g., in plasticity  \cite{MS13}, fracture mechanics \cite{Friedrich17, FSS26}, magnetoelasticity~\cite{AKM25}.  
While this shows that linearization is by now well understood for a wide range
of local models, much less is known in the nonlocal setting; see, e.g., \cite{MS24} for
a recent result in bond-based peridynamics.

In this note, we proceed in the
same spirit and study a variational model that falls within  the general framework of state-based
peridynamics \cite{S00, Sea07}. Our starting point is the theory of nonlocal
hyperelasticity, which was proposed in \cite{BCMC20Piola, BeCuMC22} to model
material behavior involving discontinuities and singularities, and has been
further developed in recent years, e.g.,~ \cite{CKS23, CKS25, BCM24, KS24, BMS25, EMS26}. 
These models are similar to those of classical hyperelasticity in that they
involve energy functionals of integral type whose densities depend on a quantity
measuring the changes in deformations $y$ of a given reference configuration
$\Omega \subset \R^n$.
 The key difference is that the classical gradient is replaced by a nonlocal
gradient $D_\rho$ associated with a suitable interaction kernel
$\rho\colon \R^n \to [0,\infty)$ that is singular at the origin. Here, we
additionally assume that $\rho$ has compact support in the ball $B_\delta$, which corresponds to a finite horizon $\delta$. 
For $\phi \in C_c^\infty(\R^n;\R^n)$, the nonlocal gradient is  then defined as 
\begin{align*}
    (D_\rho \phi)(x) = \int_{\R^n} \frac{\phi(x) - \phi(y)}{\abs{x-y}} \otimes \frac{x-y}{\abs{x-y}} \, \rho(x-y) \, dy \qquad\text{for $x\in \R^n$,}
\end{align*}
and this definition extends distributionally to weak nonlocal gradients of integrable functions; see Section~\ref{sec:nonlocalgradient}.  The admissible deformations $y$ of the model naturally lie in the nonlocal Sobolev
spaces $H^{\rho,p}(\R^n;\R^n)$ with $p \in [1,\infty]$, consisting of those $y \in L^p (\R^n; \R^n) $ with $D_{\rho} y \in (\R^n; \R^{n\times n})$. We impose boundary conditions by prescribing complementary values outside of $\Omega$, which leads
to the space $H^{\rho,p}_g(\Omega;\R^n)$ 
with $g \in H^{\rho,p}(\R^n;\R^n) \cap H^{\rho,\infty}(\R^n;\R^n)$. Note that elements of $H_{g}^{\rho,p}(\Omega;\R^n)$ can be viewed as functions on $\Omega_
\delta=\Omega+B_\delta$, since the nonlocal gradient outside $\Omega_\delta$ is fully determined by the complementary values $g$; by the same argument, the choice of $g$ outside of $\Omega_{2\delta}$ does not affect the modeling and analysis on the relevant domain. 
For the precise definitions and properties of these spaces, we refer to
Section~\ref{sec:preliminaries}.

With these notions at hand, the energies in our model of nonlocal
hyperelasticity, to be minimized among all deformations
$y \in H^{\rho,p}_g(\Omega;\R^n)$, are of the form
\begin{align}\label{Erho}
    \Ecal_\rho(y) = \int_{\Omega_\delta} W(x, D_\rho y)\,dx - \Lcal(y),
\end{align}
with $W\colon \Omega_\delta \times \R^{n\times n} \to \R$ the stored energy
density and $\Lcal\colon H^{\rho,p}(\R^n;\R^n) \to \R$ a bounded linear
functional. For example, one can take $\Lcal$ to be a standard external loading term of the
form $\Lcal(y) = \int_{\Omega_\delta} f \cdot y\,dx$ for $y \in H^{\rho,p}(\R^n;\R^n)$ with a given $f \in L^{\frac{p}{p-1}}(\R^n;\R^n)$.

The question of linearization for these nonlocal problems has already been addressed
in the literature in the case of homogeneous energies with identity boundary
conditions. More precisely, a formal linearization was carried out by Bellido, Cueto
and Mora-Corral~\cite{BCM23} for models involving finite-horizon fractional
gradients, and more recently by Bellido and Garc\'ia-S\'aez~\cite{BG26} for a
class of general nonlocal gradients.
These works show in fact that the formally derived linearized energies can be interpreted as
Eringen-type models with singular kernels, which form part of Eringen's theory of nonlocal elasticity ~\cite{E02} that is widely used in continuum mechanics
and engineering. Existence and uniqueness of solutions for such Eringen models
were studied by Evgrafov and Bellido~\cite{EB18}.

Our goal here is to provide a rigorous justification of the formal arguments
in~\cite{BCM23, BG26} by proving a corresponding nonlocal linearization result via
$\Gamma$-convergence. The latter is the natural framework for the asymptotic
analysis of variational problems; see, e.g.,~\cite{D93, BraidesGAMMA} for an
introduction and background.

Before we state our main result, some preparations are needed. Let us first
specify the assumptions on the energy functionals of the form~\eqref{Erho}. 
Following the setting of ~\cite{ADD12},
we work with a suitable single-well energy density. More precisely, we suppose that $W\colon \Omega_{\delta} \times \R^{n \times n} \to [0, \infty]$ is measurable with respect to the product of the Lebesgue $\sigma$-algebra on $\Omega_\delta$ and the Borel $\sigma$-algebra on $\R^{n\times n}$, and that, for some $p \in (1, 2]$ and for a.e.\ $x \in \Omega$, it satisfies the following conditions:

\begin{enumerate}[label=(W\arabic*)]
\item \label{W1} $W(x,F)=W(x,RF)$ for every $F \in \R^{n \times n}$ and every $R \in \SO(n)$;\\[-0.2cm]
\item $W(x,F)=0$ if $F \in \SO(n)$;\\[-0.3cm]
\item \label{W3}  $W(x,\cdot)$ is of class $C^2$ in a fixed neighborhood of $\SO(n)$ with second derivatives uniformly bounded in that neighborhood and in $x$;\\[-0.2cm]
\item \label{W4} 
$W(x,F) \geq  w_p(\dist(F, \SO(n)))$ for every $F\in \R^{n\times n}$ with
\begin{align*}
w_p:[0, \infty) \rightarrow \R, \quad t \mapsto \begin{cases}
    \frac{t^2}{2} &\text{ for } t \in [0,1), \\[0.2cm]
    \frac{t^p}{p}+\frac{1}{2}-\frac{1}{p}, &\text{ for } t \in [1, \infty).
\end{cases}
\end{align*}
\end{enumerate}

To carry out the linearization, we next introduce a suitable rescaling
of~\eqref{Erho} that makes the small-strain regime.  
As is standard, we fix a small parameter
$\epsilon > 0$, replace the external load $\Lcal$ by $\epsilon \Lcal$ and the
complementary values $g$ by $\id + \epsilon g$ truncated outside $\Omega_{2\delta}$, and write the deformation as
$y = \id + \epsilon u$ with the rescaled displacement $u$, so that the actual
displacement $\epsilon u$ is small. After rescaling the energy~\eqref{Erho}
by $\epsilon^{-2}$, we obtain the functionals

\begin{align}\label{Erhoeps}
	\Ecal_{\rho,\eps} (u) = \begin{cases} \displaystyle\frac{1}{\epsilon^2}\int_{\Omega_\delta} W(x, \mathbb{I}+ \epsilon D_{\rho} u) \dx-\Lcal(u)  & \text{ for } u\in   H^{\rho,p}_{g}(\Omega;\R^n) ,\\
		\infty &\text{ otherwise;}
	\end{cases}
\end{align}
note that here $D_{\rho} \id=\mathbb{I}$ is ensured by a normalization condition of the kernel function $\rho$, see~\ref{H5} and~\cite[Lemma 3.1]{BG26}.
For $u$ with bounded nonlocal gradient, a common argument based on Taylor expansion (see e.g. \cite[(1.4)]{DNP02}) implies,
$$ W(x, \Ibb + \epsilon  D_{\rho} u)= \frac{\eps^2}{2} D^2W(x,\Ibb)[e_{\rho}(u),e_{\rho}(u)]+ o(\eps^2) \quad \text{for a.e. $x \in \Omega_{\delta}$},$$
where $e_{\rho}(u)$ denotes the symmetrized nonlocal strain tensor, that is, the symmetrized nonlocal gradient defined by $e_{\rho}(u)= (D_\rho u)_{\rm sym} = \frac{1}{2}(D_{\rho}u+(D_{\rho}u)^{T})$. 
This motivates the definition of the linear functional
\begin{align}\label{eq:limitfunctional}
	\Ecal_{\rho, \textrm{lin}}(u)=\begin{cases}\displaystyle
		\int_{\Omega_\delta}\frac{1}{2} D^2W(x,\mathbb{I})[e_{\rho}(u),e_{\rho}(u)] \dx -\Lcal(u) & \text{ for } u \in H^{\rho,2}_{g}(\Omega;\R^n),\\
		\infty &\text{ otherwise,}
	\end{cases}
\end{align}
where $D^2W(\cdot, \Ibb)$ is a fourth-order elasticity tensor.  
We point out that $\Ecal_{\rho,\rm lin}$ is a heterogeneous version of an
Eringen-type energy, whose homogeneous counterpart is studied in~\cite{BCM23, BG26}.
For a further discussion of the limit model and its relation to models in the
applied literature, we refer to these works and the references therein.

The following nonlocal linearization result shows that $\Ecal_{\rho,\rm lin}$
is indeed the $\Gamma$-limit of $(\Ecal_{\rho,\epsilon})_{\epsilon>0}$ as
$\epsilon \to 0$, i.e., in the limit of small strains. It is the nonlocal
analogue of the linearization of Agostiniani, Dal Maso and
DeSimone~\cite{ADD12}, and hence, in particular, also of~\cite{DNP02}.

\begin{theorem}[Linearization via $\Gamma$-convergence]\label{thm:linearization_fullbc}
Let  $p \in (1,2]$ and $\Ecal_{\rho,\epsilon}, \Ecal_{\rho, \rm{lin}}:H^{\rho, p}(\R^n;\R^n)\to \R_\infty$  for $\epsilon>0$ be as in~\eqref{Erhoeps} and~\eqref{eq:limitfunctional}. 
Then $\Ecal_{\rho , \rm{lin}}$ is the $\Gamma$-limit  of the family $(\Ecal_{\rho,\epsilon})_{\epsilon>0}$ with respect to the weak topology of $H^{\rho,p}(\R^n;\R^n)$ as $\epsilon \to 0$, in formulas, 
\begin{align}\label{formula_Gamma}
\Gamma(w\text{-}H^{\rho, p})\text{-}\lim_{\epsilon \to 0} \Ecal_{\rho,\epsilon} = \Ecal_{\rho, \rm{lin}}. 
\end{align}

 Moreover, any sequence $(u_\eps)_\eps$ in $H^{\rho,p}(\R^n;\R^n)$ such that $\sup_{\eps>0} \Ecal_{\rho,\eps}(u_\eps) < \infty$ has a weakly convergent subsequence in $H^{\rho, p}(\R^n;\R^n)$.
\end{theorem}
Our methodology for proving this theorem relies on a transfer argument that reduces the problem to the classical linearization. Specifically, we carry over the $\Gamma$-convergence and compactness results established in \cite{ADD12}, which generalizes those of~\cite{DNP02}.
This approach has two main advantages over a strategy that mimics the
proofs in the local setting by replacing all classical gradients with nonlocal ones. On the one hand, it
avoids having to adapt the arguments step by step to the nonlocal setting, which results in a shorter proof.
On the other hand, similar ideas can potentially be used to deduce nonlocal
counterparts of other linearization results in the literature, e.g.,
for multiwell problems~\cite{ADLP18}, which is an interesting question for future research.  
For other types of problems in (asymptotic) variational analysis, the same transfer strategy has already been applied, for instance, 
 to derive lower semicontinuity and $\Gamma$-convergence, including homogenization and relaxation for nonlocal energies as in~\eqref{Erho}, see e.g.,~\cite{CKS23, CKS24,  KreisbeckSchonberger, KS24}.

Further, we note that a similar argument 
also shows that every recovery sequence for an element of
$H^{\rho,2}_g(\R^n;\R^n)$ converges even strongly in $H^{\rho,p}(\R^n;\R^n)$.
Together with the fact that the limit problem has a unique solution, this yields
the convergence of almost minimizers stated below. Recall that a sequence
$(u_\eps)_{\eps}$ in $H^{\rho,p}(\R^n;\R^n)$ is called a sequence of almost
minimizers of $(\Ecal_{\rho,\eps})_{\eps>0}$ if
\begin{align*}
    \Ecal_{\rho,\eps}(u_\eps) - \inf_{u \in H^{\rho,p}(\R^n;\R^n)} \Ecal_{\rho,\eps}(u) \to 0
    \quad \text{as } \eps \to 0.
\end{align*}

\begin{corollary}[Strong convergence of sequences of almost minimizers]\label{cor:strongconvofminimizers}
Let  $p \in (1,2]$ and $\Ecal_{\rho,\epsilon}, \Ecal_{\rho, \rm{lin}}:H^{\rho, p}(\R^n;\R^n)\to \R_\infty$  for $\epsilon>0$ be as in~\eqref{Erhoeps} and~\eqref{eq:limitfunctional}. 
Then any sequence of almost minimizers of $(\Ecal_{\rho, \eps})_{\eps>0}$ converges strongly to the unique minimizer of $\Ecal_{\rho, \rm{lin}}$ in $H^{\rho,p}(\R^n;\R^n)$ as $\eps\to 0$.
\end{corollary}

The remainder of this note is organized as follows.  In Section~\ref{sec:preliminaries}, we introduce the notation and describe the setting in detail. Section~3 is devoted to the proof of the $\Gamma$-convergence and compactness result in Theorem~\ref{thm:linearization_fullbc}. Finally, in Section~4, we prove the convergence of minimizers asserted in Corollary~\ref{cor:strongconvofminimizers}.

\section{Preliminaries}\label{sec:preliminaries} 

In this section, we fix the notation and recall the necessary background on the nonlocal gradient and the associated function spaces. 

\subsection{Notation}
We begin by briefly introducing the notation used throughout this note.
Let $n \in \N$ denote the dimension, and write $\R_\infty = \R \cup \{\infty\}$. 
For $x = (x_1,\ldots,x_n)$, $y = (y_1,\ldots,y_n) \in \R^n$, we denote their scalar product by $x \cdot y = \sum_{j=1}^n x_j y_j$ and their tensor product by $x \otimes y$ with $(x\otimes y)_{ij}= x_i y_j$ for $i, j=1, \ldots, n$, and we write $\abs{x}$ for the Euclidean norm of $x$. 
For any matrix $F \in \R^{n \times n}$, we denote its transpose by $F^T$, its symmetric part by $F_{\rm sym}=\frac{1}{2}(F+F^T)$, and its Frobenius norm by $\abs{F}$. Moreover, we write $\mathbb{I}$ for the identity matrix in $\R^{n\times n}$. The ball centered at $x\in \R^n$ and with radius $r>0$ is denoted by $B_r(x)$ and by $B_r$ when $x=0$. For a set $A \subset \R^n$, we write $A^c$, $\overline{A}$, and $\ONE_A$ for its complement, closure, and indicator function, respectively. 

 The support of a function $f:\R^n\rightarrow \R$ is defined as $$\supp(f)=\overline{\{x \in \R^n: f(x) \neq 0\}},$$ and the convolution of two functions $f,g:\R^n\rightarrow \R$ is given by 
$$(f*g)(x)=\int_{\R^n}f(y)g(x-y) \dy \quad \text{ for  $x \in \R^n$}, $$ whenever the integral exists. 
 For any open set $U \subset \R^n$ and any $p \in [1, \infty]$, we adopt the standard notation  $C_c^{\infty}(U)$, $L^p(U)$, and $W^{1,p}(U)$ for the spaces of compactly supported smooth functions, Lebesgue spaces, and Sobolev spaces, respectively.  The vector-valued spaces $C_c^{\infty}(U;\R^n)$, $L^p(U;\R^n)$, $L^p(U;\R^{n \times n}),$ and $W^{1,p}(U;\R^n)$ are defined componentwise. We define  $W_0^{1,p}(U;\R^n)$ via the closure of $C_c^{\infty}(U;\R^n)$ in $W^{1,p}(U;\R^n)$. If additionally, $U$ is a Lipschitz domain, this space coincides with the space of Sobolev functions with zero trace, and  
for any $h \in W^{1,p}(U;\R^n)$,  we let $
W^{1,p}_{h}(U;\R^n)= h+ W^{1,p}_{0}(U;\R^n)
$ be the space with prescribed boundary conditions.

Finally, whenever a sequence is indexed by $\epsilon$, we mean that $\epsilon$ represents some sequence $(\epsilon_k)_{k}$ of positive numbers tending to zero as $k \to \infty$.

\subsection{Nonlocal gradients and function spaces.}\label{sec:nonlocalgradient}
We now introduce the nonlocal operators and the functional-analytic framework used in the following, restricting ourselves to the results relevant here.
We follow the setting of~\cite{BMS25}, which generalizes earlier studies of Riesz fractional and finite-horizon fractional gradients, e.g.,~\cite{BCMC20Piola, BeCuMC22, Shieh1, Shieh2, COMI2019, CKS23, KreisbeckSchonberger}. Further details and results can be found in~\cite{BMS25}.

Throughout this note, we consider a kernel function $\rho:\R^n\rightarrow [0,\infty]$ satisfying the following hypotheses \ref{H0}-\ref{H4} for some $\delta_0>0$, which, in the case of integrable kernels, coincide with the assumptions in  \cite{BMS25} of the same name:
\begin{enumerate}[label=(H\arabic*)]
\setcounter{enumi}{-1}
\item \label{H0} $\rho\in L^1(\R^n)$ is radial with the property $\liminf_{t \rightarrow 0} \bar{\rho}(t)>0$, where $\bar{\rho}$ stands for the radial representative of $\rho$.
\item The function $f_{\rho}: (0, \infty) \rightarrow \R,  t \mapsto t^{n-2} \bar{\rho}(t)$ is decreasing and there is $\mu \in (0,1)$ such that $\mu f_{\rho}(\frac{t}{2})\geq f_{\rho}(t)$ for $t \in (0, \delta_0)$.
\item $f_\rho$ is smooth and there are constants $C, C_k>0$ with $k \in \N$ such that  it holds 
\begin{equation*}
-C \frac{d}{dt} f_{\rho}(t) \geq \frac{f_{\rho}(t)}{t}\qquad  \text{ and } \qquad \absB{{\frac{d^k}{\dt^k} f_{\rho}(t)}} \leq C_k  \frac{f_{\rho}(t)}{t^k}
\end{equation*}
for any $t \in (0, \delta_0)$.
\item There exists some $\sigma \in (0,1)$ such that $h_{\sigma}:(0, \infty) \rightarrow \R,  t \mapsto t^{n+ \sigma-1} \bar{\rho}(t)$ is almost decreasing on $(0, \delta_0)$.
\item \label{H4} There exists some $\gamma \in (0,1)$ such that $h_{\gamma}:(0, \infty) \rightarrow \R,  t \mapsto t^{n+ \gamma-1} \bar{\rho}(t)$ is almost increasing on $(0, \delta_0)$.
\end{enumerate}
Moreover, for some given $\delta>0$, we additionally impose that
\begin{enumerate}[label=(H\arabic*)]
\setcounter{enumi}{4}
\item \label{H5}$\rho$ is normalized such that $\norm{\rho}_{L^1(\R^n)}=n$ and satisfies $\supp(\rho) \subset \overline{B_{\delta}}$.  
\end{enumerate}
Note that the assumptions (H1)-(H4) are not used directly in this work. They rather provide access to a number of useful technical tools such as a nonlocal Poincar\'e inequality (see \cite[Theorem 4.11 (ii)]{BMS25}) and a compact embedding of the nonlocal Sobolev spaces with zero complementary values into $L^p$ (see \cite[Theorem 4.11 (ii)]{BMS25}).

We recall that the nonlocal gradient of $\phi \in C_c^\infty(\R^n;\R^n)$ is defined by
\begin{equation*}
(D_{\rho} \phi) (x)= \int\limits_{\R^n } \frac{\phi(x) - \phi(y)}{\abs{x-y}} \otimes \frac{x-y}{\abs{x-y}} \rho(x-y) \,dy  \quad \text{for $x\in \R^n$,}
\end{equation*}
and the nonlocal divergence as
\begin{equation*}
(\diverg_\rho \phi)(x)= \int\limits_{\R^n } \frac{\phi(x) - \phi(y)}{\abs{x-y}} \cdot \frac{x-y}{\abs{x-y}} \rho(x-y) \,dy \quad \text{for $x\in \R^n$.} 
\end{equation*}
The definition can be extended distributionally; indeed, we call $V \in L_{\rm{loc}}^1(\R^n;\R^{n \times n})$, the weak nonlocal gradient of $u \in L^1(\R^n;\R^n)+ L^{\infty}(\R^n;\R^n)$ and write $D_{\rho} u=V$, if
\begin{equation*}
\int\limits_{\R^n} V(x) \phi(x) \dx=- \int\limits_{\R^n} u(x) \diverg_{\rho} \phi(x) \dx \quad \text{for all $\phi \in C^{\infty}_{c}(\R^n; \R^n)$.}
\end{equation*}

With this notion of weak nonlocal gradient, one introduces the associated nonlocal Sobolev spaces as follows. 
For any $p \in [1, \infty]$, one considers the space 
\begin{equation*}
H^{\rho,p}(\R^n; \R^n)=\{u \in L^p(\R^n; \R^n): D_{\rho}u \in L^p(\R^n; \R^{n \times n}) \},
\end{equation*}
endowed with the norm $\norm{u}_{H^{\rho, p}(\R^n;\R^n)} = \norm{u}_{L^p(\R^n;\R^n)} + \norm{D_\rho u}_{L^p(\R^n;\R^{n\times n})}$.  Accordingly, a sequence $(u_k)_{k\in \N}$ converges weakly to $u$ in $H^{\rho,p}(\R^n; \R^n)$, if and only if $u_k \weakly u$ in $L^p(\R^n;\R^n)$ and $D_{\rho} u_k \weakly D_{\rho} u$ in $L^p(\R^n;\R^{n \times n})$ as $k\to \infty$. 

To impose volumetric boundary conditions, we consider the following setting. Let $\Omega \subset \R^n$ be a bounded Lipschitz domain such that $\Omega_\delta = \Omega + B_\delta$ also has Lipschitz boundary. For $p \in [1,\infty]$ and $g \in H^{\rho,\infty}(\R^n;\R^n) \cap H^{\rho,p}(\R^n;\R^n)$, we define the complementary-value space
\[
H_g^{\rho,p}(\Omega;\R^n) = \{u \in H^{\rho,p}(\R^n;\R^n) : u = g \text{ a.e.\ in } \Omega^c\}.
\]
Note that, depending on the context, one can consider elements of $H_{g}^{\rho,p}(\Omega;\R^n)$ as functions on $\Omega_
\delta$.  This is motivated by the case of zero boundary values, where the original functions can be recovered by a trivial extension. Further justification is given by the fact that for any $u \in H_{0}^{\rho,p}(\Omega;\R^n)$, it holds $D_{\rho}u =0$ a.e. on $\Omega_\delta^c$, as a consequence of $\supp(\rho)\subset \overline{B_{\delta}}$.

A key technical tool throughout this note is the following translation mechanism, which converts nonlocal gradients into classical ones see~\cite{BMS25, CKS24}.
By~\cite[Lemma~2.9]{CKS24}, for every $p \in (1,\infty)$,
\begin{equation}\label{Qcalrho_prelim}
\Qcal_\rho\colon H^{\rho,p}(\R^n;\R^n) \to W^{1,p}(\R^n;\R^n), \quad u \mapsto Q_\rho * u,
\end{equation}
is a bounded linear map with the property that
    \begin{align}\label{eq:translmechanism}
        D_{\rho} u = \nabla (\Qcal_\rho u)=\nabla (Q_\rho\ast u) \quad \text{ for } u \in H^{\rho,p}(\R^n;\R^n);
    \end{align} 
    the nonlocal potential $Q_{\rho}:\R^n \rightarrow [0, \infty)$ appearing in~\eqref{Qcalrho_prelim} is given as 
    \begin{align*}
    Q_{\rho}(x)= \int_{\abs{x}}^{\infty} \frac{\bar{\rho}(t)}{t} \dt \quad\text{ for $x\in \R^n$.}
    \end{align*}
It belongs to $L^1(\R^n)$ by~\cite[Lemma~2.5\,(iii)]{BMS25}, and since $\rho$ is compactly supported by~\ref{H5}, we have $\supp Q_\rho \subset \overline{B_\delta}$.

We close this preliminary section with a result on the transfer of boundary values. As mentioned in~\cite[Remark~4\,e)]{CKS23} for the finite-horizon fractional kernel, the operator $\Qcal_\rho$ preserves zero boundary conditions, i.e., it maps $H_0^{\rho,p}(\Omega;\R^n)$ into $W_0^{1,p}(\Omega_\delta;\R^n)$ (identifying functions with their restrictions to $\Omega_\delta$). This observation carries over to our setting, and we include a proof for the reader's convenience.

\begin{lemma}\label{le:translationboundaryvalue}
Let $p \in (1,\infty)$, and let $\Omega \subset \R^n$ be a bounded Lipschitz domain such that $\Omega_\delta$ is also Lipschitz. Then, for every $u \in H^{\rho,p}_g(\Omega;\R^n)$, it holds that $\Qcal_\rho u \in W^{1,p}_{\Qcal_\rho g}(\Omega_\delta;\R^n)$. 
\end{lemma}
\begin{proof} 
Without loss of generality, we may assume that $g = 0$, otherwise, we replace $u$ by $u - g$.
Now let $u \in H_0^{\rho,p}(\Omega;\R^n)$ and $\eta \in (0, \delta)$.  We first show that $\Qcal_\rho u\in W^{1,p}(\R^n;\R^n)$ is continuous on $\Omega_\eta^c$, where $\Omega_\eta = \Omega + B_\eta$, and then determine its trace on $\partial\Omega_\delta$ by evaluating $\Qcal_\rho u$ pointwise.

Since $u \in H_0^{\rho,q}(\Omega;\R^n)$, we have $u = 0$ a.e.\ in $\Omega^c$, and hence $(\ONE_{B_\eta} Q_\rho) * u = 0$ on 
$\Omega_\eta^c$. Consequently,
\begin{align*}
\Qcal_\rho u = Q_\rho * u = (\ONE_{B_\eta^c} Q_\rho) * u \qquad \text{on } \Omega_\eta^c.
\end{align*}
Using that $\rho \in L^1(\R^n)$ is radial and nonnegative, we obtain for every $x \in B_\eta^c$ that
\begin{align*}
Q_\rho(x) \leq \int_\eta^\infty \frac{\bar\rho(t)}{t}\,dt
= \frac{1}{\omega_n} \int_{B_\eta^c} \frac{\rho(y)}{\abs{y}^n}\,dy
\leq \frac{\norm{\rho}_{L^1(\R^n)}}{\sigma_n\,\eta^n},
\end{align*}
where $\sigma_n$ denotes the surface area of the unit sphere in $\R^n$. 
This shows $\ONE_{B_\eta^c} Q_\rho \in L^\infty(\R^n)$,
and it follows that  $\Qcal_\rho u$ is indeed continuous on $\Omega_\eta^c$ as the convolution of an integrable with an essentially bounded function.

    Finally, let $x \in \partial\Omega_\delta \subset \Omega_\eta^c$. As $\supp Q_\rho \subset \overline{B_\delta}$ and $u = 0$ a.e.\ in $B_\delta(x) \subset \Omega^c$, one has
\begin{equation*}
(\Qcal_\rho u)(x) = (Q_\rho * u)(x) = \int_{B_\delta(x)} Q_\rho(x-y)\,u(y)\,dy = 0.
\end{equation*}
As $\Qcal_\rho u$ is continuous on $\Omega_\eta^c$, its trace on $\partial\Omega_\delta$ therefore vanishes. This confirms $\Qcal_\rho u \in W^{1,p}_0(\Omega_\delta;\R^n)$.
\end{proof}

\section{Transfer of $\Gamma$-convergence and compactness results}\label{sec:Gammacompactness}

This section contains the proof of the nonlocal linearization, that is, the $\Gamma$-convergence and compactness of the family of functionals $(\Ecal_{\rho,\eps})_\eps$. The key idea is to reduce the problem to the local results of \cite{ADD12}, which generalize those of \cite{DNP02}.

We start by briefly recalling the setting. Let $p\in(1,2]$ and let $\Omega\subset\R^n$ be a bounded Lipschitz domain such that $\Omega_\delta$ is also Lipschitz for the horizon $\delta>0$. Moreover, $\rho:\R^n\to[0,\infty)$ is a kernel function satisfying~\ref{H0}-\ref{H5}, $g\in H^{\rho,p}(\R^n;\R^n)\cap H^{\rho,\infty}(\R^n;\R^n)$ represents the prescribed boundary data, and $W:\Omega_\delta\times\R^{n\times n}\to[0,\infty]$ is a measurable energy density satisfying~\ref{W1}-\ref{W4}.  The nonlocal energy functionals $\Ecal_{\rho,\eps},\Ecal_{\rho,{\rm lin}}:H^{\rho,p}(\R^n;\R^n)\to\R_\infty$ for $\eps>0$ are defined as in~\eqref{Erhoeps} and~\eqref{eq:limitfunctional}, respectively.

In addition, we introduce the analogous local energy functionals without loading terms, namely $\Fcal_\eps, \Fcal_{\mathrm{lin}}\colon W^{1,p}(\Omega_\delta;\R^n) \to \R_{\infty}$ for $\eps>0$, given by
\begin{align*}
    \Fcal_\eps(v) =  
	\begin{cases}
	\displaystyle \frac{1}{\epsilon} \int_{\Omega_{\delta}} W(x,\mathbb{I}+ \epsilon \nabla v) \dx & \text{ for } v \in W^{1,p}_{\Qcal_{\rho} g}(\Omega_{\delta};\R^n), \\
	\infty  & \text{ otherwise,}
	\end{cases}
\end{align*}
and
\begin{align*}
    	\Fcal_{\textrm{lin}}(v)= 
	\begin{cases}
	\displaystyle \int_{\Omega_\delta}\frac{1}{2} D^2W(x,\mathbb{I})[e(v),e(v)] \dx  & \text{ for } v \in W^{1,2}_{\Qcal_{\rho} g}(\Omega_{\delta};\R^n), \\
		\infty  & \text{ otherwise,}
	\end{cases} 
\end{align*}
where $e(v) = (\nabla v)_{\mathrm{sym}} = \tfrac{1}{2}\bigl(\nabla v + (\nabla v)^{\top}\bigr)$ denotes the symmetrized gradient of $v$.

The main property of these local functionals, which several of the subsequent proofs rely on, is the $\Gamma$-convergence of \cite[Theorem~2.4]{ADD12} (see also \cite[Proposition~2.4]{DNP02}), saying that $\Fcal_{\mathrm{lin}}$ is the $\Gamma$-limit of $(\Fcal_{\eps})_{\eps>0}$ with respect to the weak topology of $W^{1,p}(\Omega_\delta;\R^n)$ as $\eps \to 0$, in formulas, 
\begin{align}\label{eq:localGammalimit}
\Gamma(w\text{-}W^{1,p}(\Omega_\delta;\R^n))\text{-}\lim_{\eps \to 0} \Fcal_{\eps} = \Fcal_{\rm{lin}}.
\end{align}

We also note the following relation between the local and nonlocal energy functionals, induced by the transfer operator $\Qcal_\rho\colon H^{\rho, p}(\R^n;\R^n)\to W^{1,p}(\R^n; \R^n)$; cf.~\eqref{eq:translmechanism} and Lemma~\ref{le:translationboundaryvalue}.
 In fact, for all $u\in H_g^{\rho,p}(\Omega;\R^n)$, it holds that
\begin{align}\label{eq:ErhoepsbyFeps}
\Ecal_{\rho, \eps}(u) = \Fcal_\eps(\Qcal_\rho u) - \Lcal(u) \quad\text{and}\quad \Ecal_{\rho, \rm lin}(u) = \Fcal_{\rm lin}(\Qcal_\rho u) - \Lcal(u),
\end{align}
where the second identity follows from $e_\rho(u) = e(\Qcal_\rho u)$, a direct consequence of~\eqref{eq:translmechanism} and the definition of the symmetrized gradient. 
 Here and in what follows, we identify functions $v\in W^{1,p}(\R^n;\R^n)$ with their restrictions to $\Omega_\delta$, i.e., $v = v|_{\Omega_\delta}$, whenever the context requires it and without indicating this explicitly in the notation. This applies in particular to $\Qcal_\rho u$ for $u\in H^{\rho,p}(\R^n;\R^n)$ in the expressions $\Fcal_\eps(\Qcal_\rho u)$ and $\Fcal_{\rm lin}(\Qcal_\rho u)$ in~\eqref{eq:ErhoepsbyFeps}.

\medskip
As the first ingredient towards proving Theorem \ref{thm:linearization_fullbc}, we show that equicoercivity of the functionals $(\Fcal_\eps)_\eps$ carries over to $(\Ecal_{\rho, \eps})_\eps$. 
\begin{proposition}[Transfer of equicoercivity]\label{prop:compactnesstranslation} 
Let $\Ecal_{\rho,\epsilon}, \Ecal_{\rho, \rm{lin}}:H^{\rho, p}(\R^n;\R^n)\to \R_\infty$  for $\epsilon>0$  be as in~\eqref{Erhoeps} and~\eqref{eq:limitfunctional}. 
 Then, every sequence $(u_\eps)_{\eps}$ in $ H^{\rho,p}(\R^n;\R^n) $ satisfying 
\begin{align}\label{eq:functionalboundedness}
{\sup_{\eps>0} \Ecal_{\rho,\epsilon}(u_\eps)< \infty}\end{align}
has a subsequence converging weakly in $H^{\rho, p}(\R^n;\R^n)$. 
\end{proposition}

\begin{proof}
Throughout this proof, $C>0$ denotes a generic constant independent of $\eps$, which may change from line to line.

Let $(u_\eps)_{\eps}$ in $ H^{\rho,p}(\R^n;\R^n)$ satisfy~\eqref{eq:functionalboundedness}, then $u_\eps\in H_g^{\rho,p}(\Omega;\R^n)$ for all $\eps$.   
By \cite[Theorem~2.3]{ADD12}, there is a constant $C>0$ such that every sequence $(v_\eps)_\eps$ in $W^{1,p}_{\Qcal_\rho g}(\Omega_{\delta};\R^n)$ satisfies
\begin{align}\label{est_Feps88}
\norm{\nabla v_{\eps}}_{L^p(\Omega_{\delta};\R^{n \times n})}^p \leq C\big(1+ \Fcal_{\eps}(v_\eps)\big)
\end{align}
for all sufficiently small $\eps>0$. In what follows, we assume that $\eps$ is small enough so that~\eqref{est_Feps88} holds. 
Since $\Qcal_{\rho} u_\eps\in W^{1,p}_{\Qcal_\rho g}(\Omega_{\delta};\R^n)$ for all $\eps$ due to~\eqref{Qcalrho_prelim} and Lemma~\ref{le:translationboundaryvalue}, we can apply~\eqref{est_Feps88} to $v_\eps = \Qcal_{\rho} u_\eps$. Together with~\eqref{eq:translmechanism}, \eqref{eq:ErhoepsbyFeps}, the boundedness of the linear functional $\Lcal$, and the uniform bound from~\eqref{eq:functionalboundedness}, this yields
\begin{align*}
\norm{D_{\rho} u_{\eps}}_{L^p(\Omega_{\delta};\R^{n \times n})}^p
&= \norm{\nabla (\Qcal_\rho u_{\eps})}_{L^p(\Omega_{\delta};\R^{n \times n})}^p
\leq C\big(1+ \Fcal_{\eps}(\Qcal_\rho u_{\eps})\big)\\
&= C\big(1+ \Ecal_{\rho,\eps}(u_{\eps}) - \Lcal(u_{\eps})\big)
\leq C\big(1 + \norm{u_{\eps}}_{H^{\rho,p}(\R^n;\R^n)}\big). 
\end{align*}

Finally, we combine this with the nonlocal Poincar\'e inequality (see \cite[Theorem 4.11\,(i)]{BMS25}) to obtain
\begin{align*}
     \norm{u_{\eps}}_{H^{\rho,p}(\R^n; \R^n)}^p\leq  
C\Bigl(\norm{D_{\rho}u_{\eps}}_{L^p(\Omega_{\delta};\R^{n \times n})}^p+\norm{g}_{H^{\rho,p}(\R^n; \R^{n})}^p \Bigr)
     \leq C \bigl(1+ \norm{u_{\eps}}_{H^{\rho,p}(\R^n;\R^n)}\bigr).
\end{align*} 
Since $p>1$, we conclude that the sequence $(u_\eps)_\eps$ is bounded in $H^{\rho,p}(\R^n;\R^n)$, which, by reflexivity, yields the stated equicoercivity of $(\Ecal_{\rho,\eps})_\eps$.
\end{proof}

 Next we now turn to transferring the local $\Gamma$-limit to the nonlocal setting, starting with the $\liminf$-inequality.

\begin{proposition}[Transfer of the liminf-inequality]\label{prop:Gammaliminf} Let $\Ecal_{\rho,\epsilon}, \Ecal_{\rho, \rm{lin}}:H^{\rho, p}(\R^n;\R^n)\to \R_\infty$  for $\epsilon>0$ be as in~\eqref{Erhoeps} and~\eqref{eq:limitfunctional}.  
It holds for every $(u_\eps)_{\eps}$ and $u$ in $H^{ \rho,p}_{g}(\Omega;\R^n)$ with $u_\eps \weakly u$ in $H^{\rho, p}(\R^n;\R^n)$ that
\begin{align*}
\liminf_{\eps \rightarrow 0} \Ecal_{\rho, \epsilon}(u_\eps) \geq \Ecal_{\rho, \rm{lin}}(u).    
\end{align*}
\end{proposition}

\begin{proof}
Let $(u_k)_{k}$ be a sequence in $H^{ \rho,p}_{g}(\Omega;\R^n)$ such that $u_k \weakly u$ in $H^{\rho, p}(\R^n;\R^n)$ for some $u \in H^{\rho,p}_{g}(\Omega;\R^n)$, and let $(\epsilon_k)_{k}$ be a sequence of positive numbers tending to zero as $k \rightarrow \infty$. 
Since $\Qcal_{\rho}$ from \eqref{eq:translmechanism} is linear and bounded, it follows (after restriction to $\Omega_\delta$) that
\begin{align*}
 \Qcal_{\rho} u_k \weakly \Qcal_{\rho} u \quad\text{in $W^{1,p}(\Omega_\delta;\R^n)$.}
\end{align*}

By the liminf inequality associated with the $\Gamma$-limit~\eqref{eq:localGammalimit},  every sequence $(v_k)_{k\in\N}$ in $W^{1,p}(\Omega_{\delta};\R^n)$ with $v_k \weakly v$ in $W^{1,p}(\Omega_\delta;\R^n)$ satisfies
\begin{align}\label{eq:localliminfinequality}
 \liminf_{k \to \infty}\Fcal_{\epsilon_k}(v_k) \geq \Fcal_{\mathrm{lin}}(v).
\end{align}
Hence,  applying \eqref{eq:localliminfinequality} with $v_k = \Qcal_\rho u_k$ and using \eqref{eq:ErhoepsbyFeps}, we conclude together with the weak continuity of $\Lcal$ as a bounded linear functional on $H^{\rho,p}(\R^n;\R^n)$ that
\begin{align*}
 \liminf_{k \to \infty}\Ecal_{\rho, \epsilon_k}(u_k)
 = \liminf_{k \to \infty}\Fcal_{\epsilon_k}(\Qcal_\rho u_k) - \lim_{k\to\infty}\Lcal(u_k)
 \geq \Fcal_{\mathrm{lin}}(\Qcal_\rho u) - \Lcal(u)
 = \Ecal_{\rho,\mathrm{lin}}(u).
\end{align*}
This proves the claim.
\end{proof}

To prove the remaining limsup inequality, we construct a recovery sequence, which even converges strongly.

\begin{proposition}[Construction  of recovery sequences]\label{prop:Gammalimsup}
Let $\Ecal_{\rho,\epsilon}, \Ecal_{\rho, \rm{lin}}:H^{\rho, p}(\R^n;\R^n)\to \R_\infty$  for $\epsilon>0$ be as in~\eqref{Erhoeps} and~\eqref{eq:limitfunctional}. 
For every $u \in H^{\rho,2}_g(\Omega;\R^n)$, there exists $(u_\eps)_{\eps} $ in $H^{\rho,p}_g(\Omega;\R^n)$ such that $u_\eps \to u$ in $H^{\rho,p}(\R^n;\R^n)$ as $\eps \rightarrow 0$ and 
$$ \lim_{\eps \rightarrow 0} \Ecal_{\rho,\eps}(u_\eps)= \Ecal_{\rho,\rm{lin}}(u).$$
\end{proposition}

\begin{proof} 
Let $(\epsilon_k)_{k}$ be a sequence of positive numbers tending to zero as $k \rightarrow \infty$. We show in Step~1 that for any $u \in H_g^{\rho, 2}(\Omega;\R^n) \cap H^{\rho, \infty}(\R^n;\R^n)$, the constant sequence is a recovery sequence for $u$, and then conclude in Step~2 by a density argument.\smallskip

\textit{Step~1.} Fix $u \in H_g^{\rho, 2}(\Omega;\R^n) \cap H^{\rho, \infty}(\R^n;\R^n)$. Using  \eqref{eq:translmechanism} (applied with $p=2$),  we infer that $\nabla (\Qcal_{\rho}u)=D_\rho u \in L^{\infty}(\R^n;\R^{n \times n})$. Further, the Young convolution inequality combined with $Q_{\rho} \in L^1(\R^n)$ implies  ~$\Qcal_{\rho} u=Q_{\rho}*u \in L^{\infty}(\R^n;\R^n) $. Hence, we have $\Qcal_\rho u\in W^{1, \infty}(\R^n;\R^n)$, so that $\Qcal_{\rho} u \in W_{\Qcal_\rho g}^{1,\infty}(\Omega_{\delta};\R^n)$ (after identification with the restriction to $\Omega_\delta$). 
 Since every function in $W_{\Qcal_\rho g}^{1,\infty}(\Omega_{\delta};\R^n)$ has a constant recovery sequence for the local $\Gamma$-limit~\eqref{eq:localGammalimit} according to ~\cite[(4.2)]{ADD12}, we deduce
\begin{align*}
\lim_{k\rightarrow \infty} \Fcal_{\eps_k} (\Qcal_\rho u) 
= \Fcal_{\rm lin} (\Qcal_\rho u).
\end{align*} 
This, together with~\eqref{eq:ErhoepsbyFeps}, shows $\lim_{k\rightarrow \infty} \Ecal_{\rho,\eps_k} (u)= \Ecal_{\rho,\rm{lin}} (u)$,  as claimed.\smallskip

\textit{Step~2. } We assume without loss of generality that $g = 0$, as the general case follows by applying the same argument to the shifted functions $u - g$.
 Now consider any $u \in H^{\rho,2}_0(\Omega;\R^n)$. 
Since $\rho$ has compact support, \cite[Proposition~2.6\,$(i)$]{BMS25} gives $D_\rho\phi \in C_c^\infty(\R^n;\R^n)$ for every $\phi \in C_c^\infty(\R^n;\R^n)$, and hence $C_c^\infty(\Omega;\R^n) \subset H_0^{\rho,2}(\Omega;\R^n) \cap H^{\rho,\infty}(\R^n;\R^n)$, upon identifying functions with their zero extension.   Recalling from~\cite[Theorem~3.9\,$(iii)$]{BMS25} that $C_c^\infty(\Omega;\R^n)$ is dense in $H_0^{\rho,2}(\Omega;\R^n)$, there exists a sequence $(\phi_j)_{j}$ in $C_c^\infty(\Omega;\R^n)$ with 
\begin{align}\label{approxvarphik}
\phi_j \to u \qquad \text{in $H^{\rho,2}(\R^n;\R^n)$ as $j \to \infty$.}
\end{align}

  Let us observe that, due to \ref{W3}, there exists a constant $C > 0$ such that
\begin{align*}
\biggl|\int_{\Omega_\delta} D^2W(x,\mathbb{I})[e_\rho(v), e_\rho(w)]\,dx\biggr| 
&\leq C \norm{e_\rho(v)}_{L^2(\R^n;\R^{n\times n})} \norm{e_\rho(w)}_{L^2(\R^n;\R^{n\times n})} \\ 
&\leq C \norm{D_\rho v}_{L^2(\R^n;\R^{n\times n})} \norm{D_\rho w}_{L^2(\R^n;\R^{n\times n})}
\end{align*}
for all $v, w \in H^{\rho,2}(\R^n;\R^n)$; hence, the bilinear form $\Bcal_\rho \colon H_0^{\rho,2}(\Omega;\R^n) \times H_0^{\rho,2}(\Omega;\R^n) \to \R$ defined by
\begin{equation}\label{eq:defaW}
\Bcal_\rho(v,w) = \frac{1}{2}\int_{\Omega_\delta} D^2W(x,\mathbb{I})[e_\rho(v), e_\rho(w)]\,dx
\end{equation}
is bounded. Consequently, the map $H_0^{\rho,2}(\Omega;\R^n) \to \R$, $v \mapsto \Bcal_\rho(v,v)$, is continuous as well. Together with the continuity of $\Lcal$ and~\eqref{approxvarphik}, we then conclude 
\begin{align}\label{eq:limitElinphi}
\lim_{j\to\infty} \Ecal_{\rho,\mathrm{lin}}(\phi_j)
= \lim_{j\to\infty} \Bcal_\rho(\phi_j,\phi_j) - \lim_{j\to\infty} \Lcal(\phi_j)
= \Bcal_\rho(u,u) - \Lcal(u)
= \Ecal_{\rho,\mathrm{lin}}(u).
\end{align}

 Finally, a standard diagonalization argument (cf.~Attouch's lemma \cite[Lemma 1.15 and Corollary 1.16]{At84}) merging~\eqref{eq:limitElinphi} with Step~1 for each $\varphi_j$ gives rise to a sequence $(j_k)_{k}$ of indices such that 
\begin{align*}
\lim_{k\rightarrow \infty} \Ecal_{\rho,\epsilon_k}(\phi_{j_k})=\lim_{k\rightarrow \infty} \Ecal_{\rho,\rm{lin}}(\phi_{j_k})=\Ecal_{\rho,\rm{lin}}(u).
\end{align*}  
It remains to show that $\phi_{j_k} \to u$ in $H^{\rho,p}(\R^n;\R^n)$ as $k\to \infty$. Since $p \leq 2$, this follows from the convergence in $H^{\rho,2}(\R^n;\R^n)$, accounting for the boundary conditions.
\end{proof}

\begin{proof}[Proof of Theorem~\ref{thm:linearization_fullbc}]
The proof now follows directly from the preceding results. The $\liminf$-inequality from Proposition~\ref{prop:Gammaliminf} and the existence of recovery sequences from Proposition~\ref{prop:Gammalimsup} together yield the stated $\Gamma$-convergence~\eqref{formula_Gamma}, while the equicoercivity of the functionals $(\Ecal_{\rho, \eps})_\eps$ is established in Proposition~\ref{prop:compactnesstranslation}.
\end{proof}

\section{Proof of convergence of minimizers}

This final section is concerned with the proof of the strong convergence of minimizing sequences from Corollary~\ref{cor:strongconvofminimizers}. Throughout, we work in the same setting as in Section~\ref{sec:Gammacompactness}. 
Our proof strategy mimicks the standard argument from the local theory (see Theorems~2.2 in \cite{DNP02, ADD12}), which relies on the strong convergence of recovery sequences and the uniqueness of the minimizer of the limit functional. We begin with the former, which we again prove by a transfer argument, that is, by reducing the problem to its local analogue~\cite[Theorem~2.5]{ADD12}.

\begin{lemma}[Strong convergence of recovery sequences]\label{le:strongconvergencerecovery}
Let $\Ecal_{\rho,\epsilon}, \Ecal_{\rho, \rm{lin}}:H^{\rho, p}(\R^n;\R^n)\to \R_\infty$  for $\epsilon>0$ be as in~\eqref{Erhoeps} and~\eqref{eq:limitfunctional}. 
For each $u \in H^{\rho,2}_g(\Omega;\R^n)$, every recovery sequence $(u_k)_{k}$ for $u$ associated with the $\Gamma$-limit \eqref{formula_Gamma} converges strongly in $H^{\rho,p}(\R^n;\R^n)$.
\end{lemma}

\begin{proof}
Let $(\epsilon_k)_{k}$ be a sequence of positive numbers with $\epsilon_k \to 0$ as $k \to \infty$. Given $u \in H^{\rho,2}_g(\Omega;\R^n)$,  consider a sequence $(u_k)_{k}$ in $H^{\rho,p}(\R^n;\R^n)$ such that $u_k \weakly u$ in $H^{\rho,p}(\R^n;\R^n)$ and
\begin{align*}\label{eq:recoveryseq}
\lim_{k \to \infty} \Ecal_{\rho,\epsilon_k}(u_k) = \lim_{k \to \infty} \Ecal_{\rho,\textrm{lin}}(u).
\end{align*}
Since $\Ecal_{\rho,\textrm{lin}}(u) < \infty$, it follows that $\Ecal_{\rho,\epsilon_k}(u_k) < \infty$, and hence $u_k \in H^{\rho,p}_g(\Omega;\R^n)$, for all sufficiently large $k \in \N$.

The continuity of the operator $\Qcal_{\rho}$ from \eqref{Qcalrho_prelim} implies $\Qcal_{\rho} u_k \weakly \Qcal_{\rho} u$ in $W^{1,p}(\Omega_\delta;\R^n)$. Moreover, \eqref{eq:ErhoepsbyFeps} and the continuity of $\Lcal$ yield
\begin{equation*}
\lim_{k \to \infty} \Fcal_{\epsilon_k}(\Qcal_{\rho} u_k) = \lim_{k \to \infty} \bigl( \Ecal_{\rho,\epsilon_k}(u_k) - \Lcal(u_k) \bigr) = \Ecal_{\rho,\textrm{lin}}(u) - \Lcal(u) = \Fcal_{\textrm{lin}}(u).
\end{equation*}
This shows that $(\Qcal_{\rho} u_k)_{k}$ is a recovery sequence for $\Qcal_{\rho} u$ associated with the  $\Gamma$-limit \eqref{eq:localGammalimit}.
We thus conclude from the strong convergence result in~\cite[Theorem~2.5]{ADD12} that 
\begin{align*}
\Qcal_{\rho} u_k \to \Qcal_{\rho} u \quad \text{ in $W^{1,p}(\Omega_{\delta};\R^n)$.}
\end{align*}

Now, accounting for the boundary conditions, the transfer mechanism in~\eqref{eq:translmechanism} allows us to deduce that
\begin{align*}
D_{\rho} u_k =\nabla (\Qcal_{\rho}u_k) \rightarrow \nabla (\Qcal_{\rho}u)=D_{\rho} u \quad \text{ in $L^p(\R^n;\R^{n \times n})$  as $k \to \infty$}. 
\end{align*}
Combining this with the observation that $u_k \to u$ in $L^{p}(\R^n;\R^n)$, which follows from the compact embedding of $H^{\rho,p}_{0}(\Omega;\R^n)$ into $L^{p}(\R^n;\R^n)$ (see~\cite[Theorem~4.11\,(ii)]{BMS25}), we arrive at the desired strong convergence $u_k \to u$ in $H^{\rho,p}(\R^n;\R^n)$ as $k \to \infty$.
\end{proof}

We now show that the functional $\Ecal_{\rho,\textrm{lin}}$ defined in \eqref{eq:limitfunctional} admits a unique minimizer. To do so, it suffices to adapt the argument of \cite[Theorem~4.3]{BG26}, which relies on the Lax-Milgram lemma and a nonlocal Korn inequality, to the case of heterogeneous energy densities. Although the proof is similar, we include it here for the reader's convenience.

\begin{lemma}[Existence of unique minimizer of $\Ecal_{\rho,\textrm{lin}}$]\label{le:uniqueminlimitmodel}

The functional $\Ecal_{\rho,\rm{lin}}\colon H^{\rho,p}(\R^n;\R^n) \to \R_\infty$  defined in~\eqref{eq:limitfunctional} has a unique minimizer. 
\end{lemma}
\begin{proof} 
Using the estimate 
	\begin{align*}
		D^2W(x,\mathbb{I})[F_{\rm{sym}},F_{\rm{sym}}] \geq \abs{F_{\rm{sym}}}^2 \quad \text{ for all $F \in \R^{n\times n}$,}
	\end{align*}
    see \cite[(2.4)]{ADD12},
	along with the nonlocal Korn inequality in  \cite[Proposition 4.2]{BG26},
	we deduce that there is a constant $C>0$ such that
	\begin{align*}
		\int_{\Omega_{\delta}} D^2W(x,\mathbb{I})[e_{\rho}(v),e_{\rho}(v)] \dx \geq C \norm{D_{\rho} v}_{L^2(\Omega_{\delta};\R^{n \times n})}^2  \quad \text{ for all } v \in H_{0}^{\rho,2} (\Omega;\R^n).
	\end{align*}
    Therefore, the bilinear form $\mathcal{B}_{\rho}\colon H_{0}^{\rho,2}(\Omega;\R^n) \times H_{0}^{\rho,2}(\Omega;\R^n) \to \R$ introduced in~\eqref{eq:defaW} is coercive and, in view of~\ref{W3}, also bounded. 
 These observations allow us to apply Lax-Milgram lemma (see e.g. \cite[Corollary 5.8]{B11}) to the bilinear form $\Bcal_{\rho}$ and the bounded linear functional $l:H_{0}^{\rho,2} (\Omega;\R^n) \rightarrow \R$ given by
	\begin{align*}
		l(v) = \frac{1}{2}\int_{\Omega_{\delta}} D^2W(x,\mathbb{I})[e_{\rho}(v),e_{\rho}(g)] \dx+\frac{1}{2}\int_{\Omega_{\delta}} D^2W(x,\mathbb{I})[e_{\rho}(g),e_{\rho }(v)] \dx+ \Lcal(v),
	\end{align*}
   which yields the existence of a unique minimizer of $v\mapsto \mathcal{B}_{\rho}(v,v)-l(v)$ in $H_{0}^{\rho,2} (\Omega;\R^n)$. 
    
    As a consequence of the computation that for any $v \in H^{\rho,2}_0(\Omega;\R^n)$ and $u=v+g$,
	\begin{align*}
		\int_{\Omega_{\delta}} D^2W(x,\mathbb{I})[e_{\rho}(u),e_{\rho}(u)] \dx&=\int_{\Omega_{\delta}} D^2W(x,\mathbb{I})[e_{\rho}(v),e_{\rho}(v)] \dx +\int_{\Omega_{\delta}} D^2W(x,\mathbb{I})[e_{\rho}(v),e_{\rho}(g)] \dx\\&+\int_{\Omega_{\delta}} D^2W(x,\mathbb{I})[e_{\rho}(g),e_{\rho}(v)] \dx+\int_{\Omega_{\delta}} D^2W(x,\mathbb{I})[e_{\rho}(g),e_{\rho}(g)] \dx,
	\end{align*}
	we observe that  $v$ 
    minimizing $v\mapsto\mathcal{B}_{\rho}(v, v)-l(v)$ is equivalent to $u=v+g$
    minimizing $\Ecal_{\rho, \rm{lin}}$. This completes the proof.
\end{proof}

\begin{proof}[Proof of Corollary \ref{cor:strongconvofminimizers}] 
Let $(\epsilon_k)_{k}$ be a sequence of positive reals such that $\epsilon_k \to 0$ as $k \to \infty$ and let $(u_k)_{k}$ be a sequence of almost minimizers of $(\Ecal_{\rho,\epsilon_k})_{k}$. Further, let  $u_0$ 
denote the unique minimizer of $\Ecal_{\rho, \rm lin}$.

 The $\Gamma$-convergence and equicoercivity of the functionals $(\Ecal_{\rho,\epsilon})_{\epsilon>0}$ derived in Theorem~\ref{thm:linearization_fullbc} combined with the uniqueness of the minimizer of $\Ecal_{\rho,\rm{lin}}$ from Lemma \ref{le:uniqueminlimitmodel} imply  in view of~\cite[Corollary 7.24 and Proposition 8.16]{D93} 
 that  $u_k \weakly u_{0}$ in $H^{\rho,p}(\R^n;\R^n)$ as $k\to \infty$ and
\begin{align*}
 \lim_{k\rightarrow \infty} \Ecal_{\rho,\eps_k}(u_k)=\min_{u \in H^{\rho,p}(\R^n;\R^n)} \Ecal_{\rho,\rm{lin}}(u) 
 = \Ecal_{\rho,\rm{lin}}(u_0).
\end{align*}
In particular, $(u_k)_{k}$ is also a recovery sequence for $u_0$, so that applying Lemma \ref{le:strongconvergencerecovery} yields even the strong convergence $u_k \to u_0$ in $H^{\rho,p}(\R^n;\R^n)$.
\end{proof}

\section*{Acknowledgements}

The authors would like to thank Jos\'e Carlos Bellido and Guillermo Gac\'ia-S\'aez for valuable discussions. 
This research was funded by the AEI-DFG project ``Nonlocal Gradients in Variational Analysis and Materials Modelling: Limits, Kernels, Boundaries"
through the  Agencia Estatal de Investigaci\'on (AEI) project number PCI2024-155023-2 and the Deutsche Forschungsgemeinschaft (DFG)  
project number 541520348.

\bibliographystyle{abbrv}
\bibliography{bibliography}

\end{document}